\documentclass[reqno, 12pt]{amsart}
\usepackage{a4wide, amsfonts, amssymb, amsmath, amsthm, mathrsfs, enumerate, enumitem, setspace, url}
\usepackage[utf8]{inputenc}
\usepackage{tikz-cd} %diagrams
	\usetikzlibrary{cd}
\usepackage{color}
\usepackage{hyperref} %hyperlinks
	\definecolor{darkred}{rgb}{0.5,0,0}
	\definecolor{darkgreen}{rgb}{0,0.5,0}
	\definecolor{darkblue}{rgb}{0,0,0.5}
	\hypersetup{linkcolor=darkblue,filecolor=darkgreen,urlcolor=darkred,citecolor=darkblue}

\usepackage{comment}

\usepackage[T2A,T1]{fontenc}
\DeclareSymbolFont{cyrillic}{T2A}{cmr}{m}{n}
\DeclareMathSymbol{\Sha}{\mathalpha}{cyrillic}{216}

\usepackage[toc,page]{appendix}

\makeatletter
\newcommand{\leqmode}{\tagsleft@true}
\newcommand{\reqmode}{\tagsleft@false}
\makeatother

\theoremstyle{plain}
\newtheorem{theorem}{Theorem}[section]
\newtheorem*{theorem*}{Theorem}
\newtheorem{proposition}[theorem]{Proposition}
\newtheorem{lemma}[theorem]{Lemma}
\newtheorem{corollary}[theorem]{Corollary}

\theoremstyle{remark}
\newtheorem{remark}[theorem]{Remark}
\newtheorem*{acknowledgements}{Acknowledgements}

\theoremstyle{definition}
\newtheorem{definition}[theorem]{Definition}

\numberwithin{equation}{section}

\newcommand{\ZZ}{\mathbb{Z}}
\newcommand{\QQ}{\mathbb{Q}}

\newcommand{\FF}{\mathbb{F}}

\renewcommand{\AA}{\mathbb{A}}
\newcommand{\PP}{\mathbb{P}}
\newcommand{\GG}{\mathbb{G}}

\newcommand{\bfa}{\mathbf{a}} \newcommand{\bfA}{\mathbf{A}}

\newcommand{\sF}{\mathcal{F}}
\newcommand{\sO}{\mathcal{O}}

\DeclareMathOperator{\HH}{H}
\DeclareMathOperator{\Gal}{Gal}
\DeclareMathOperator{\Pic}{Pic}
\DeclareMathOperator{\Br}{Br}

\DeclareMathOperator{\id}{id}

\renewcommand{\epsilon}{\varepsilon}

\begin{document}
\onehalfspacing
\title[]{Genus one fibrations and vertical Brauer elements on del Pezzo surfaces of degree 4}

\author{Vladimir Mitankin}
	\address{Leibniz Universit\"at Hannover, Institut f\"ur Algebra,
  Zahlentheorie und Diskrete Mathematik, Welfengarten 1, 30167
  Hannover, Germany}
	\email{mitankin@math.uni-hannover.de}

\author{Cec\'ilia Salgado}
	\address{Mathematics Department \\
		Bernoulli Institute\\
	Rijksuniversiteit Groningen\\
		The Netherlands
		}
	\email{c.salgado@rug.nl}
\date{\today}

\begin{abstract}
We consider a family of smooth del Pezzo surfaces of degree four and study the geometry and arithmetic of a genus one fibration with two reducible fibres for which a Brauer element is vertical. 
\end{abstract}  
\maketitle
\setcounter{tocdepth}{1}
\tableofcontents

\section{Introduction}
A del Pezzo surface of degree four $X$ over a number field $k$ is a smooth projective surface in $\mathbb{P}^4$ given by the complete intersection of two quadrics defined over $k$. They are the simplest class of del Pezzo surfaces that have a positive dimensional moduli space and for which interesting arithmetic phenomena happen.
Indeed, del Pezzo surfaces of degree at least 5 with a $k$-point are birational to $\mathbb{P}^2_{k}$ and, in particular, have a trivial Brauer group. They satisfy the Hasse Principle and weak approximation. The Brauer group $\Br X = \HH_{\text{\'{e}t}}^2(X,\mathbb{G}_m)$ of $X$ is a birational invariant which encodes important arithmetic information such as failures of the Hasse principle and weak approximation via the Brauer--Manin obstruction. We refer the reader to \cite[\S8.2]{P17} for an in-depth description of this obstruction. The image $\Br_0 X$ of the natural map $\Br k \rightarrow \Br X$ does not play a r\^{o}le in detecting a Brauer--Manin obstruction and thus one can consider the quotient $\Br X / \Br_0 X$ instead of $\Br X$. We say that $X$ has a trivial Brauer group when this quotient vanishes. 

In contrast to del Pezzo surfaces of higher degree, the Hasse principle may fail for del Pezzo surfaces of degree four \cite{JS17}. Yet, they form a tractable class.  Colliot-Th\'el\`ene and Sansuc conjectured in \cite{CTS80} that all failures of the Hasse principle and weak approximation are explained by the Brauer-Manin obstruction. This is established conditionally for certain families (\cite{Wit07}, \cite{VAV14}). 

In \cite{VAV14} V\'arily-Alvarado and Viray proved that del Pezzo surfaces of degree four that are everywhere locally soluble have a vertical Brauer group. In particular, given a Brauer element $\mathcal{A}$, they show that there is a genus one fibration $g$, with at most two reducible fibres, for which $\mathcal{A}\in g^*(\Br(k(\mathbb{P}^1)))$. The aim of this paper is to study this fibration in detail for a special family of quartic del Pezzo surfaces which we investigated from arithmetic and analytic point of view in \cite{MS20}.

Let $\bfa = (a_0, \dots, a_4)$ be a quintuple with coordinates in the ring of integers $O_k$ of $k$. Define $X_\bfa \subset \PP_k^4$ by the complete intersection
\begin{equation}
	\label{eq:dP4 main}
	\begin{split}
		x_0x_1 - x_2x_3 = 0, \\
		a_0x_0^2 + a_1x_1^2 + a_2x_2^2 + a_3x_3^2 + a_4x_4^2 = 0 
	\end{split}
\end{equation}
and we shall assume from now on that $X_\bfa$ is smooth. The latter is equivalent to $(a_0a_1 - a_2a_3)\prod_{i = 0}^4 a_i \neq 0$. This altogether gives the following family of interest to us in this article:
\[	
	\sF = \{X_\bfa \mbox{ as in } \eqref{eq:dP4 main} \ : \ \bfa \in O_k^5 \mbox{ and } (a_0a_1 - a_2a_3)\prod_{i = 0}^4 a_i \neq 0\}. 
\]

There are numerous reasons behind our choice of this family. Firstly, surfaces in $\sF$ admit two distinct conic bundle structures, making their geometry and hence their arithmetic considerably more tractable. Moreover, for such surfaces the conjecture of Colliot-Th\'{e}l\`{e}ne and Sansuc is known to hold unconditionally \cite{CT90}, \cite{Sal86}. Secondly, our surfaces can be thought of as an analogue of diagonal cubic surfaces as they also satisfy the interesting equivalence of $k$-rationality and trivial Brauer group. This is shown in Lemma~\ref{rational} which is parallel to \cite[Lem.~1.1]{CTKS}. 

Our aim is to take advantage of the two conic bundle structures present in the surfaces to give a thorough description of a genus one fibration with two reducible fibres for which a Brauer element is vertical. More precisely, after studying the  action of the absolute Galois group on the set of lines on the surfaces, we show that the two reducible fibres are of type $I_4$ and that the field of definition of the Mordell--Weil group of the associated elliptic surface depends on the order of the Brauer group modulo constants which in our case is 1, 2 or 4 \cite{Man74}, \cite{SD93}. The presence of the two conic bundle structures plays an important r\^ole forcing a bound on the degree and shape of the Galois group of the field of definition of the lines. We show in Theorem~\ref{thm:MWBrauer} that surfaces with Brauer group of size 2 are such that the genus one fibration only admits a section over a quadratic extension of $k$, while those with larger Brauer group, namely of order 4, have a section for the genus 1 fibration already defined over $k$.
\begin{theorem}
	\label{thm:MWBrauer}
	Let $X_\bfa \in \sF$ and let $\mathcal{E}$ be the genus one fibration on $X_\bfa$ described in \S4.2. Then the following hold.
		\begin{enumerate}[label=\emph{(\roman*)}]
			\item If $\Br X_\bfa / \Br_0 X_\bfa \simeq \mathbb{Z}/2\mathbb{Z}$, then the genus 1 fibration $\mathcal{E}$ is an elliptic fibration i.e., admits a section, over a quadratic extension. Moreover, it admits a section of infinite order over a further quadratic extension. The Mordell--Weil group of $\mathcal{E}$ is fully defined over at most a third quadratic extension.
			\item If $\Br X_\bfa / \Br_0 X_\bfa \simeq (\mathbb{Z}/2\mathbb{Z})^2$, then $\mathcal{E}$ is an elliptic fibration with a 2-torsion section and a section of infinite order over $k$. Moreover, the full Mordell--Weil group of $\mathcal{E}$ is defined over a quadratic extension.
		\end{enumerate}
\end{theorem}

Not surprisingly, this is in consonance with the bounds obtained in our earlier paper \cite[\S1]{MS20} when $k = \QQ$, as surfaces with Brauer group of size 2 are generic in the family while those with larger Brauer group are special.

%In attempt to keep the level of generality when permitted, we advert the reader that the letter $k$ is used at different instances to denote a global field or, more specifically, a number field. We work from the most general concept to the least, and make the distinction clear throughout the text. 

This paper is organized as follows. Section~\ref{theconics} contains some generalities on quartic del Pezzo surfaces that admit two conic bundles. There we also describe the two conic bundles on the surfaces of interest to us.  Section~\ref{sec:lines} is devoted to the study of the action of the absolute Galois group on the set of lines on $X_\bfa$. We have also included there a description of the Brauer elements using lines, by means of results of Swinnerton-Dyer, giving the tools to, in Section~\ref{thegenus1}, describe a genus one fibration with exactly two reducible fibres for which a Brauer element is vertical. 

\begin{acknowledgements}
	We would like to thank Martin Bright, Yuri Manin and Bianca Viray for useful discussions. We are grateful to the Max Planck Institute for Mathematics in Bonn and the Federal University of Rio de Janeiro for their hospitality while working on this article. Cec\'ilia Salgado was partially supported by FAPERJ grant E-26/202.786/2019, Cnpq grant PQ2 310070/2017-1 and the Capes-Humboldt program.
\end{acknowledgements}

\section{Two conic bundles}\label{theconics}
Let $X$ be a quartic del Pezzo surface over a number field $k$. From this point on we assume that $X$ is $k$-minimal and moreover that it admits a conic bundle structure over $k$. It follows from \cite{Isk71} that there is a second conic bundle structure on $X$. In this context, given a line $L\subset X$ then $L$ plays simultaneously the r\^ole of a fibre component and of a section depending on the conic bundle considered.

Fix a separable closure $\bar{k}$ of $k$. In what follows we analyse the possible orbits of lines under the action of the absolute Galois group $\Gal(\bar{k}/k)$ when $\Br X \neq \Br_0 X$ in the light of the presence of two conic bundle structures over $k$. Firstly, we recall \cite[Prop.~13]{BBFL07} that tells us the possible sizes of the orbits of lines. In the statement of this proposition the authors consider a quartic del Pezzo surface over $\QQ$ but its proof establishes the result for a del Pezzo surface of degree four over any number field.

\begin{lemma}\label{lem:BBFL}[{\cite[Prop.~13]{BBFL07}}]
	\label{lem:sizes}  
	Let $X$ be a del Pezzo surface of degree four over $k$. Assume that $\Br X / \Br_0 X$ is not trivial. Then the $\Gal(\bar{k}/k)$-orbits of lines in $X$ are one of the following:
	\[
		(2,2,2,2,2,2,2,2), (2,2,2,2,4,4), (4,4,4,4), (4,4,8), (8,8).
	\]
\end{lemma}

\begin{remark}\label{rmk:conic_orbit}
Recall that we have assumed that $X$ is minimal. In particular, every orbit contains at least two lines that intersect. Since each conic bundle is defined over $k$ and the absolute Galois group acts on the Picard group preserving intersection multiplicities, we can conclude further that each orbit is formed by conic bundle fibre(s). In other words, if a component of a singular fibre of a conic bundle lies in a given orbit, then the other component of the same fibre also lies in it. 
 \end{remark}
%This together with Lemma \ref{lem:sizes} gives us:

%\begin{proposition}
%Let $X$ be a del Pezzo surface of degree four defined over a number field $k$. Assume moreover that it is $k$-minimal and admits a conic bundle structure over $k$. Then the orbits of lines in $X$ under $\Gal(\bar{k}/k)$ are as follows.
%\begin{enumerate}[label=\emph{(\roman*)}]
%\item They all have size 2 and each orbit is formed by a singular fibre of a conic bundle on $X$.
%\item There are four orbits of size 2 each formed by a singular fibre of a conic bundle on $X$ and two orbits of size 4 each formed by two singular fibres of a second conic bundle on $X$.
%\item There are four orbits of size 4, each formed by two singular fibres of one of the conic bundles on $X$.
%\item There are two orbits of size 4 each formed by two singular fibres of a conic bundle on $X$ and one orbit of size 8 formed by all singular fibres of a second conic bundle on $X$.
%\item There are two orbits of size 8 each formed by all singular fibres of a conic bundle on $X$.
%\end{enumerate} 
%\end{proposition}

\subsection{A special family with two conic bundles}
We now describe the two conic bundle structures over $k$ in the del Pezzo surfaces given by \ref{eq:dP4 main}. It suffices to consider $\FF(1, 1, 0) = \PP(\sO_{\PP^1}(1)\oplus\sO_{\PP^1}(1)\oplus\sO_{\PP^1})$ which one can think of as $((\AA^2 \setminus 0) \times (\AA^3 \setminus 0))/ \GG_m^2$, where $\GG_m^2$ acts on $(\AA^2 \setminus 0) \times (\AA^3 \setminus 0)$ as follows:
\[
	(\lambda, \mu) \cdot (s, t; x, y, z) = (\lambda s, \lambda t; \frac{\mu}{\lambda}x, \frac{\mu}{\lambda} y, \mu z).
\]
The map $\FF(1, 1, 0) \rightarrow \PP^4$ given by $(s, t; x, y, z) \mapsto (sx: ty: tx: sy: z)$ defines an isomorphism between $X_\bfa$ and
\begin{equation}
	\label{eqn:conic bundle}
	(a_0s^2 + a_2t^2)x^2 + (a_3s^2 + a_1t^2)y^2 + a_4z^2 = 0 \subset \FF(1, 1, 0).
\end{equation}
A conic bundle structure $\pi_1 : X_\bfa \rightarrow \PP^1$ on $X_\bfa$ is then given by the projection to $(s, t)$.

Similarly, one obtains $\pi_2 : X_\bfa \rightarrow \PP^1$ via $(s, t; x, y, z) \mapsto (tx: sy: ty: sx: z)$. It gives a second conic bundle structure on $X_\bfa$ as shown by the equation
\begin{equation}
	\label{eqn:conic bundle2}
	(a_0t^2 + a_3s^2)x^2 + (a_1s^2 + a_2t^2)y^2 + a_4z^2 = 0 \subset \FF(1, 1, 0).
\end{equation}

This puts us in position to refine Lemma \ref{lem:BBFL} upon restricting our attention to surfaces in the family $\mathcal{F}$.

\begin{lemma}\label{lem:refine}
Let $X$ be a $k$-minimal del Pezzo surface of degree four described by equation \eqref{eq:dP4 main}. Then the $\Gal(\bar{k}/k)$-orbits of lines in $X$ are one of the following:
	\[
		(2,2,2,2,2,2,2,2), (2,2,2,2,4,4), (4,4,4,4).
	\]
\end{lemma}
\begin{proof}
We only have to eliminate the possibility of orbits of size 8.  One can see readily from \ref{eqn:conic bundle} and \ref{eqn:conic bundle2} that each line on $X$ is defined over at most a biquadratic extension of $k$.
 \end{proof}

\section{Lines and Brauer elements}
\label{sec:lines}
%In what follows we will analyse how the conditions of Proposition \ref{prop:BrXconic} on the coefficients $\bfa=(a_0:\cdots : a_4)$ reflect on the arithmetic of the lines on the del Pezzo surface $X_\bfa$. The results in this section hold over an arbitrary number field, but for the sake of consistency we work over $\QQ$. 

Following Swinnerton-Dyer \cite{SD99} we detect the double fours that give rise to Brauer classes. Firstly, we show that a del Pezzo surface of degree 4 given by \eqref{eq:dP4 main} has a trivial Brauer group if and only if it is rational over the ground field (see Lemma~\ref{rational}). In particular, no $k$-minimal del Pezzo surface of degree 4 given by \eqref{eq:dP4 main} has a trivial Brauer group. We take a step further after Lemma \ref{lem:refine} and note that for a del Pezzo surface of degree 4 with a conic bundle structure the sizes of the orbits of lines are determined by the order of the Brauer group (but, of course, not vice-versa as a surface with eight pairs of conjugate lines can have both trivial or non-trivial Brauer group for example). On the other hand, if one assumes that the Brauer group is non-trivial then the size of the orbits does determine that of the Brauer group (see Lemma~\ref{sizeorbit}). Moreover, given a non-trivial Brauer element, we describe in detail a genus one fibration with exactly two reducible fibres as in \cite{VAV14} for which this element is vertical. We obtain a rational elliptic surface by blowing up four points, namely two singular points of fibres of the conic bundle \eqref{eqn:conic bundle} together with two singular points of fibres of the conic bundle \eqref{eqn:conic bundle2}. The field of definition of the Mordell--Weil group of the elliptic fibration is determined by the size of the Brauer group of $X_\bfa$. In general, it is fully defined over a biquadratic extension. We also show that the reducible fibres are both of type $I_4$.

\subsection{Conic bundles and lines}

Let $X_\bfa$ be given by \eqref{eq:dP4 main}. Then it admits two conic bundle structures given by \eqref{eqn:conic bundle} and \eqref{eqn:conic bundle2}. Each conic bundle has two pairs of conjugate singular fibres with Galois group $(\ZZ/2\ZZ)^2$ acting on the 4 lines that form each of the two pairs. The intersection behavior of the lines on $X_\bfa$ is described in Figure \ref{intersectionlines}. Together, these 8 pairs of lines give the 16 lines on $X_\bfa$. 

%The singular fibers of (\ref{eqn:conic bundle}) are located above $\frac{s}{t}=\pm \sqrt{-\frac{a_2}{a_0}}, \pm \sqrt{-\frac{a_1}{a_3}}$. The first two are given by the pair of lines $(\sqrt{-a_0a_4d}y-z)(\sqrt{-a_0a_4d}y+z)$, while the last two are given by $(\sqrt{a_3a_4d}x-z)(\sqrt{a_3a_4d}x+z)$.

%The singular fibers of \ref{eqn:conic bundle2} occur above $\frac{s}{t}=\pm\sqrt{-\frac{a_0}{a_3}},\pm\sqrt{-\frac{a_1}{a_2}}$. The first two are given by the two pairs of lines $(\sqrt{a_3a_4d}y-z)(\sqrt{a_3a_4d}y+z)$, while the last two are given by $(\sqrt{-a_1a_4d}x-z)(\sqrt{-a_1a_4d}x+z)$.

We now assign a notation to work with the lines. Given $i\in \{1, \cdots, 4\}$, the union of two lines $L_i^{+}$ and $L_i^{-}$ will denote the components of a singular fibre of the conic bundle \eqref{eqn:conic bundle}. Similarly, the union of two lines $M_i^{+}$ and $M_i^{}$ will denote the singular fibres of the conic bundle \eqref{eqn:conic bundle2}. More precisely, using the variables $(x_0: x_1: x_2: x_3 : x_4)$ to describe the conic bundles, we have the following

\leqmode
\begin{align*} 
\tag{$L_1^{\pm}$}   x_0x_1=x_2x_3= -\sqrt{-\frac{a_2}{a_0}}, \quad x_4&=\pm \sqrt{\frac{d}{-a_0a_4}}x_1, \\
\tag{$L_2^{\pm}$}  x_0x_1=x_2x_3= \sqrt{-\frac{a_2}{a_0}}, \quad x_4&=\pm\sqrt{\frac{d}{-a_0a_4}}x_1,\\
\tag{$L_3^{\pm}$}  x_0x_1=x_2x_3= -\sqrt{-\frac{a_1}{a_3}}, \quad x_4&=\pm \sqrt{\frac{d}{a_3a_4}}x_2, \\
 \tag{$L_4^{\pm}$}  x_0x_1=x_2x_3=  \sqrt{-\frac{a_1}{a_3}}, \quad x_4&=\pm \sqrt{\frac{d}{a_3a_4}}x_2, \\
 \tag{$M_1^{\pm}$}  x_0x_1=x_2x_3= -\sqrt{-\frac{a_0}{a_3}}, \quad x_4&=\pm \sqrt{\frac{d}{a_3a_4}}x_2,\\
  \tag{$M_2^{\pm}$}  x_0x_1=x_2x_3= \sqrt{-\frac{a_0}{a_3}}, \quad x_4&=\pm \sqrt{\frac{d}{a_3a_4}}x_2,\\
\tag{$M_3^{\pm}$}  x_0x_1=x_2x_3= -\sqrt{-\frac{a_2}{a_1}}, \quad x_4&=\pm \sqrt{\frac{d}{-a_0a_4}}x_1,\\
  \tag{$M_4^{\pm}$}  x_0x_1=x_2x_3= \sqrt{-\frac{a_2}{a_1}}, \quad x_4&=\pm \sqrt{\frac{d}{-a_0a_4}}x_1.
\end{align*}
\reqmode

One can readily determine the intersection behavior of these lines, which we describe in Lemma \ref{lemma:fours}. We also take the opportunity to identify fours and double fours defined over small field extensions. Recall that a \emph{four} in a del Pezzo surface of degree 4 is a set of four skew lines that do not all intersect a fifth one. A \emph{double four} is four together with the four lines that meet three lines from the original four (\cite[Lemma10]{SD93}).

\begin{lemma}\label{lemma:fours}
  Let $i, j,k,l \in \{1,\cdots, 4\}$ with $j\neq i$. Consider $L_i^{+}, L_i^{-}, M_i^{+}$ and $M_i^{-}$  as above. Then 
\begin{enumerate}[label=\emph{(\alph*)}]
 \item $L_i^{+}$ intersects $L_i^{-}, M_i^{-}$ and $ M_j^{+}$, while  $L_i^{-}$ intersects $L_i^{+}, M_i^{+}$, and  $M_j^{-}$.
 \item $M_i^{+}$ intersects $M_i^{-}, L_i^{-}$ and  $L_j^{+}$, while $M_i^{-}$ intersects $M_i^{+}, L_i^{+}$ and  $L_j^{-}$.
 \item The lines $L_i^{+},L_j^{+},M_k^{-},M_l^{-}$ and the lines $L_i^{-},L_j^{-},M_k^{+},M_l^{+}$, with $i+j \equiv k+l\equiv  3 \bmod 4$, form two fours defined over the same field extension $L/ k$ of degree at most 2. Together they form a double four defined over $k$. 
 \end{enumerate}

\end{lemma}
\begin{proof}
 Statements (a) and (b) are obtained by direct calculations. For the line $L_1$, for instance, one sees readily that it intersects $L_1^{-}, M_1^{-}, M_2^{+},M_3^{+}$ and $M_4^{+}$ respectively at the points $(-\sqrt{\frac{-a_2}{a_0}}:0:1:0:0),(-\sqrt{\frac{-a_2}{a_0}}:-\sqrt{\frac{-a_0}{a_3}}:1:-\sqrt{\frac{a_2}{a_3}}:-\sqrt{\frac{d}{{a_4a_3}}}),(-\sqrt{\frac{-a_2}{a_0}}:\sqrt{\frac{-a_0}{a_3}}:1:\sqrt{\frac{a_2}{a_3}}:\sqrt{\frac{d}{{a_4a_3}}}),(-\sqrt{\frac{-a_2}{a_0}}:-\sqrt{\frac{-a_2}{a_1}}:1:\frac{a_2}{\sqrt{a_0a_1}}:-\sqrt{\frac{d a_2}{a_4a_0a_1}})$ and $(-\sqrt{\frac{-a_2}{a_0}}:\sqrt{\frac{-a_2}{a_1}}:1:-\frac{a_2}{\sqrt{a_0a_1}}:\sqrt{\frac{d a_2}{a_4a_0a_1}})$.
 Part (c) follows from (a) and (b). To see that one of such fours is defined over an extension of degree at most 2, note that each subset $\{L_i^{+},L_j^{+}\}$ and $\{M_k^{-},M_l^{-}\}$ is defined over the same extension of degree 2. For instance, taking $i=1,j=2,k=3$ and $l=4$, we see that the four is defined over $k(\sqrt{-a_0a_4d})$. The double four is defined over $k$ since both $\{L_i^{+},L_j^{+},L_i^{-},L_j^{-}\}$ and $\{M_k^{+},M_l^{+},M_k^{-},M_l^{-}\}$ are Galois invariant sets. 
\end{proof}

\begin{figure}[h]
	\label{intersectionlines}
 \[
  \begin{tikzpicture}[inner sep=0,x=25pt,y=15pt,font=\footnotesize]

   %\fill [orange!20] (0,-27) rectangle (1,27);  

   \draw[line width=2pt, white] (-4,-5) -- (-3,5);

   \draw[very thick, white]  (-3,-5) -- (-4,5);

   \draw[very thick, white] (-2,-5) -- (-1,5);

      \draw[very thick, white] (-1,-5) -- (-2,5);

       \draw[very thick, white] (1,-5) -- (2,5);

   \draw[very thick, white] (2,-5) -- (1,5);

   \draw[very thick, white] (3,-5) -- (4,5);

      \draw[very thick, white] (4,-5) -- (3,5);

   \draw  (-4,-5) -- (-3,5);
      \node at (-4.25,5.25) {$L_1^{+}$};

   \draw  (-3,-5) -- (-4,5);
         \node at (-3.25,5.25) {$L_1^{-}1$};

   \draw (-2,-5) -- (-1,5);
      \node at (-2.25,5.25) {$L_2^{+}$};

      \draw (-1,-5) -- (-2,5);
      \node at (-1.25,5.25) {$L_2^{-}$};

       \draw (1,-5) -- (2,5);
      \node at (1.25,5.25) {$L_3^{+}$};

   \draw (2,-5) -- (1,5);

        \node at (2.25,5.25) {$L_3^{-}$};

   \draw (3,-5) -- (4,5);
      \node at (3.25,5.25) {$L_4^{+}$};

      \draw (4,-5) -- (3,5);
      \node at (4.25,5.25) {$L_4^{-}$};

    \draw[line width=2pt, white] (-5,-4) -- (5,-3);

   \draw[line width=2pt, white] (-5,-3) -- (5,-4);

   \draw[line width=2pt, white] (-5,-2) -- (5,-1);

      \draw[line width=2pt, white] (-5,-1) -- (5,-2);

       \draw[line width=2pt, white] (-5,1) -- (5,2);

   \draw[line width=2pt, white] (-5,2) -- (5,1);

   \draw[line width=2pt, white] (-5,3) -- (5,4);

      \draw[line width=2pt, white] (-5,4) -- (5,3);

      \draw (-5,-4) -- (5,-3);
      \node at (5.25,-4.25) {$M_1^{-}$};

   \draw (-5,-3) -- (5,-4);
      \node at (5.25,-3.25) {$M_1^{+}$};

   \draw (-5,-2) -- (5,-1);
      \node at (5.25,-1.75) {$M_2^{-}$};

      \draw (-5,-1) -- (5,-2);
      \node at (5.25,-0.75) {$M_2^{+}$};

       \draw (-5,1) -- (5,2);
      \node at (5.25,1.25) {$M_3^{-}$};

   \draw (-5,2) -- (5,1);
      \node at (5.25,2.25) {$M_3^{+}$};

   \draw (-5,3) -- (5,4);
      \node at (5.25,3.25) {$M_4^{-}$};

      \draw (-5,4) -- (5,3);
      \node at (5.25,4.25) {$M_4^{+}$};

%\draw[densely dotted] (-3.5) -- (0.5,-27);

%   \draw[fill=white] (0,0) circle (2.5pt);

   %\draw[fill=white] (0,0) circle (2.5pt);

   \filldraw (-3.5,0) circle (2pt);

    %  \filldraw (-3.5,0) circle (2pt);

   \filldraw (1.5,0) circle (2pt);

   \filldraw (-1.5,0) circle (2pt);

      \filldraw (3.5,0) circle (2pt);

      \filldraw (0, -3.5) circle (2pt);

   \filldraw (0,1.5) circle (2pt);

   \filldraw (0,-1.5) circle (2pt);=0,

      \filldraw (0,3.5) circle (2pt);

         \filldraw (-3.9,-3.9) circle (2pt);

                  \filldraw (-3.2,-3.2) circle (2pt);

                  \filldraw (-3.3,-1.82) circle (2pt);

                  \filldraw (-3.6,1.15) circle (2pt);

                  \filldraw (-3.81,3.15) circle (2pt);

                  \filldraw (-1.12,-3.63) circle (2pt);

                  \filldraw (-1.34,-1.34) circle (2pt);

                  \filldraw (-1.62,1.3) circle (2pt);

                  \filldraw (-1.82,3.28) circle (2pt);

                   \filldraw (1.83,-3.3) circle (2pt);

                  \filldraw (1.65,-1.35) circle (2pt);

                  \filldraw (1.37,1.37) circle (2pt);

                  \filldraw (3.2,3.2) circle (2pt);

                  \filldraw (1.16,3.63) circle (2pt);

              \filldraw (3.83,-3.1) circle (2pt);

                 \filldraw (3.63,-1.15) circle (2pt);

                   \filldraw (3.3,1.8) circle (2pt);

                    \filldraw (-1.82,-3.3) circle (2pt);

                   \filldraw (1.15,-3.67) circle (2pt);

       \filldraw (3.15,-3.8) circle (2pt);

       % \filldraw (1.17,-3.35) circle (2pt);

       \filldraw (1.35,-1.64) circle (2pt);

                 \filldraw (1.68,1.68) circle (2pt);

          \filldraw (1.84,3.32) circle (2pt);

           \filldraw (3.33,-1.81) circle (2pt);

                  \filldraw (3.62,1.17) circle (2pt);

       \filldraw (3.87,3.87) circle (2pt);

       \filldraw (-3.6,-1.12) circle (2pt);

                  \filldraw (-1.68,-1.7) circle (2pt);

                         \filldraw (-3.3,1.82) circle (2pt);

      \filldraw (-3.12,3.8) circle (2pt);

      \filldraw (-1.35,1.62) circle (2pt);

      \filldraw (-1.13,3.64) circle (2pt);
  \end{tikzpicture}
 \] 
 \caption{The lines on $X_\bfa$ and their intersection behaviour. The intersection points of pairs of lines are marked with $\bullet$.}
\end{figure}
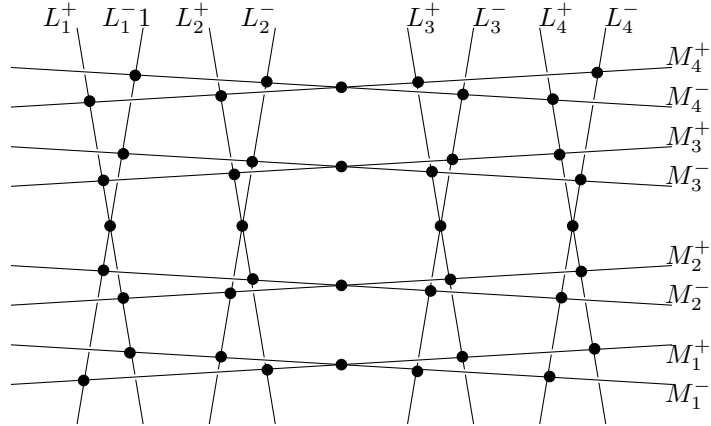

Among the 40 distinct fours on a del Pezzo surface of degree 4, the ones that appear in the previous lemma are special. More precisely, given a four as in Lemma \ref{lemma:fours} such that its field of definition has degree $d\in \{1,2\}$, the smallest degree possible among such fours, then any other four is defined over an extension of degree at least $d$. 
\begin{definition}
 Given a four as in Lemma \ref{lemma:fours} part (c), we call it a \emph{minimal four} if the field of definition of its lines  has the smallest degree among such fours.
\end{definition}

For the sake of simplicity and completion we state a result proved in \cite[Prop. 2.2]{MS20} that determines the Brauer group of $X_\bfa$ in terms of the coefficients $\bfa=(a_0,\cdots, a_4)$. We remark that the statement of the proposition below does not require that the set of adelic points $X_\bfa (\bfA_k)$ of $X_\bfa$ is non-empty and that the proof presented in \cite{MS20} works over an arbitrary number field $k$.

\begin{proposition}
	\label{prop:BrXconic}
	Let $(*)$ denote the condition that  $-a_0a_4d \notin k(\sqrt{-a_0a_2})^{*2}$, $-a_1a_4d \notin k(\sqrt{-a_1a_3})^{*2}$ and that one of $-a_0a_2$, $-a_1a_3$ or $a_0a_1$ is not in $k^{*2}$.	Then we have
	\[
		\Br X_{\bfa} / \Br_0 X_{\bfa} =
		\begin{cases}
			(\ZZ/2\ZZ)^2 &\mbox{if } a_0a_1, a_2a_3, -a_0a_2 \in k^{*2} \mbox{ and } -a_0a_4d \not\in k^{*2}, \\
			\ZZ/2\ZZ &\mbox{if } (*),\\ 
			\{\id\} &\mbox{otherwise.}
			%\mbox{if } -a_0a_4d \in \QQ(\sqrt{-a_0a_2})^{*2} \mbox{ or } -a_1a_4d \in \QQ(\sqrt{-a_1a_3})^{*2}.
		\end{cases}
	\]
\end{proposition} 

Recall the definition of a rank of a fibration \cite{Sko96}, which (as in \cite{FLS18}) for the sake of clarity to be distinguished from the Mordell--Weil rank or the Picard rank we call complexity here. That is the sum of the degrees of the fields of definition of the non-split fibres. It is clear that the conic bundles in $X_\bfa$ have complexity at most four. This allows us to obtain in our setting the following lemma.

\begin{lemma}\label{rational}
 Let $k$ be a number field and $X_\bfa$ given by \eqref{eq:dP4 main}.  Assume that $X_\bfa(\bfA_{k}) \neq \emptyset$. Then $X_\bfa$ is $k$-rational if and only if $\Br X_\bfa = \Br k$.
 \end{lemma}

 \begin{proof}
   The \emph{if} implication holds for any $k$-rational variety since $\Br X_\bfa$ is a birational invariant. To prove the non-trivial direction, we make use of \cite{KM17} which shows that conic bundles of complexity at most 3 with a rational point are rational. Firstly, note that the assumption $X_\bfa(\bfA_{k})\neq \emptyset$ implies that $\Br k$ injects into $\Br X_\bfa$. If $\Br X_\bfa / \Br k$ is trivial, then either $-a_0a_4d \in k(\sqrt{-a_0a_2})^{*2}$ or $-a_1a_4d \in k(\sqrt{-a_1a_3})^{*2}$ by Proposition~\ref{prop:BrXconic}.  Thus the complexity of the conic bundle $\pi_1$ is at most 2. It remains to show that $X_\bfa$ admits a rational point. This follows from the independent work in \cite{CT90} and \cite{Sal86} which show that the Brauer--Manin obstruction is the only obstruction to the Hasse principle for conic bundles with 4 degenerate geometric fibres. There is no such obstruction when $\Br X_\bfa / \Br k$ is trivial. Under the assumption $X(\bfA_{k})\neq \emptyset$ we conclude that $X_\bfa$ admits a rational point and hence is rational.
 \end{proof}

 \begin{remark}
  Lemma \ref{rational} is parallel to \cite[Lem.~1]{CTKS} which deals with diagonal cubic surfaces whose Brauer group is trivial. Moreover, a simple exercise shows that in our case, if the Brauer group is trivial, then the surface is a blow up of a Galois invariant set of four points in the ruled surface $\PP^1 \times \PP^1$, while the diagonal cubic satisfying the hypothesis of \cite[Lem.~1]{CTKS} is a blow up of an invariant set of six points in the projective plane. The Picard group over the ground field of the former is of rank four while that of the latter has rank three.   
 \end{remark}

\subsection{Brauer elements and double fours}

The following two results of Swinnerton-Dyer allow one to describe Brauer elements via the lines in a double four, and to determine the order of the Brauer group.

The first result is contained in \cite[Lem. 1, Ex. 2]{SD99}. 

\begin{theorem}
	\label{doublefour-Brauer}
	Let $X$ be a del Pezzo surface of degree 4 over a number field $k$ and $\alpha$ a non-trivial element of $\Br X$. Then $\alpha$ can be represented by an Azumaya algebra in the following way: there is a double-four defined over $k$ whose constituent fours are not rational but defined over $k(\sqrt{b})$, for some non-square $b \in k$. Further, let $V$ be a divisor defined over $k(\sqrt{b})$ whose class is the sum of the classes of one line in the double-four and the classes of the three lines in the double-four that meet it, and let $V'$ be the Galois conjugate of $V$ . Let $h$ be a hyperplane section of S. Then the $k$-rational divisor $D = V + V'-2h$ is principal, and if $f$ is a function whose divisor is $D$ then $\alpha$ is represented by the quaternion algebra $(f,b)$.
\end{theorem}

The following can be found at \cite[Lem.~11]{SD93}.

\begin{lemma}
	\label{doublefour-sizeBrauer}
	The Brauer group $\Br X$ cannot contain more than three elements of order 2. It
	contains as many as three if and only if the lines in $X$ can be partitioned into four
	disjoint cohyperplanar sets $T_i ,\, \, i=1,.., 4$, with the following properties:
	\begin{enumerate}[label=\emph{(\arabic*)}]
		\item the union of any two of the sets $T_i$ is a double-four;
		\item each of the $T_i$ is fixed under the absolute Galois group;
		\item if $\gamma$ is half the sum of a line $\lambda$ in some $T_i$, the two lines in the same $T_i$ that meet $\lambda$, and one other line that meets $\lambda$, then no such $\gamma$ is in $\Pic X \otimes \mathbb{Q} + \Pic \bar{X}$. 
	\end{enumerate}
\end{lemma}

We proceed to analyse how the conic bundle structures in $X_\bfa$ and the two results above can be used to describe the Brauer group of $X_\bfa$. 

\subsection{The general case}
We first describe the general case, i.e., on which there are four Galois orbits of lines of size four.

\begin{proposition}\label{prop: generalfours}
Let $X_\bfa \in \sF$ and assume that $\bfa$ satisfies hypothesis $(*)$ of Proposition \ref{prop:BrXconic}. Then there are exactly two distinct double fours on $X_\bfa$ defined over $k$ with constituent fours defined over a quadratic extension. In other words, there are exactly 4 minimal fours which pair up in a unique way to form two double fours defined over $k$. 
\end{proposition}

\begin{proof}
Part (c) of Lemma~\ref{lemma:fours} tells us that the minimal fours are given by the double four formed by the fours $\{L_1^{+},L_2^{+},M_3^{-},M_4^{-}\}, \{L_1^{-},L_2^{-},M_3^{+},M_4^{+}\}$ and that formed by $\{L_3^{+},L_4^{+},M_1^{-},M_2^{-}\}$ and  $\{L_3^{-},L_4^{-},M_1^{+},M_2^{+}\}$. By the hypothesis, each four is defined over a quadratic extension and the two double fours are defined over $k$. The hypothesis on the coefficients of the equations defining $X_\bfa$ also imply that any other double four is defined over a non-trivial extension of $k$. For instance, consider a distinct four containing $L_1^{+}$. For a double four containing this four to be defined over $k$, we need that the second four contains $L_1^{-}$ and that one of the fours contains $L_2^{+}$ and the other $L_2^{-}$. The hypothesis that each four is defined over a degree two extension gives moreover that $L_2^{+}$ is in the same four as $L_1^+$ and hence, due to their intersecting one of the lines, $M_1^{+}$ and $M_2^{+}$ cannot be in the same four. We are left with $L_3^{+},L_4^{+},M_3^{+},M_4^{+}$ and their conjugates. But if $L_3^{+}$ is in one of the fours then $L_3^{-}$ would be in the other four. This is impossible as neither $L_3^{+}$ nor $L_3^{-}$ intersect $L_1^{-}$ or $L_2^{-}$, and each line on a double four intersects three lines of the four that do not contain it. 
\end{proof}

\begin{corollary}
Let $X_\bfa$ be as above. Then $\Br X_\bfa/\Br_0 X_{\bfa}$ is of order 2.
\end{corollary}
\begin{proof}
This is a direct consequence of  Proposition~\ref{prop: generalfours} together with Theorem~\ref{doublefour-Brauer}.
\end{proof}
We shall now allow further assumptions on the coefficients of $X_\bfa$ to study how they influence the field of definition of double fours and hence the Brauer group.

\subsection{Trivial Brauer group}

 Suppose that one of $-a_0a_4d, -a_1a_4d, a_2a_4d, a_3a_4d$ is in $k^{*2}$. Assume, to exemplify, that $-a_0a_4d$ is a square. Consider the conic bundle structure given by \eqref{eqn:conic bundle}. Then the lines $L_1^{+}$ and $L_2^{+}$ are conjugate and, clearly, do not intersect. Indeed, they are components of distinct fibres of \eqref{eqn:conic bundle}. Contracting them we obtain a del Pezzo surface of degree 6. If $X_\bfa$ has points everywhere locally, the same holds for the del Pezzo surface of degree 6 by Lang--Nishimura \cite{Lang}, \cite{Nishimura}. As the latter satisfies the Hasse principle, it has a $k$-point. In particular, $X_\bfa$ is rational, which gives us an alternative proof of Lemma~\ref{rational}.

  \subsection{Brauer group of order four}
  
 For the last case, assume that $a_0a_1, a_2a_3, -a_0a_2 \in k^{*2}, -a_0a_4d, -a_1a_4d, a_2a_4d,  a_3a_4d \not\in k^{*2}$. We produce two double fours that give distinct Brauer classes. Firstly note that all the singular fibres of the two conic bundles are defined over $k$. In particular, their singularities are $k$-rational points and thus there is no Brauer--Manin obstruction to the Hasse principle and $\Br_0 X_\bfa = \Br k$. Moreover, every line is defined over a quadratic extension, but no pair of lines can be contracted since each line intersects its conjugate. Secondly, note that since $-a_0a_2$ is a square, thus $k(\sqrt{-a_0a_4d})=k(\sqrt{a_2a_4d})$. We have the double four as above, given by $L_1^{+},L_2^{+},M_3^{-},M_4^{-}$ and the correspondent intersecting components, and a new double four given by $\{L_1^{+},L_3^{+},M_2^{-},M_4^{-}\} ,\{L_2^{+},L_4^{+},M_1^{-},M_4^{-}\}$, which under this hypothesis is formed by two \emph{minimal} fours. 

 The Picard group of $X_\bfa$ is generated by $L_1^{+},L_2^{+}, L_3^{+},L_4^{+}$, a smooth conic and a section, say $M_1^{+}$ of the conic fibration (\ref{eqn:conic bundle}). We can apply Lemma~\ref{doublefour-sizeBrauer} with $T_i=\{L_i^{+},L_i^{-}, M_i^{+}, M_i^{-}\}$ to check that in this case the Brauer group has indeed size four.
 
\begin{lemma}
	\label{sizeorbit}
	Let $X_\bfa$ be as in \eqref{eq:dP4 main}. Assume that $X_\bfa$ does not contain a pair of skew conjugate lines or, equivalently, $X_\bfa$ is not $k$-rational. Then the following hold:
	\begin{enumerate}[label=\emph{(\roman*)}]
	 	\item $\# \Br X_\bfa /\Br_0 X_\bfa =4$ if and only if the set of lines on $X_\bfa$ has orbits of size \newline $(2,2,2,2,2,2,2,2)$.
 		\item $\# \Br X_\bfa /\Br_0 X_\bfa =2$ if and only if the set of lines on $X_\bfa$ has orbits of size \newline $(2,2,2,2,4,4)$ or $(4,4,4,4)$.
	\end{enumerate}
\end{lemma}
\begin{proof}
 This is an application of \cite[Lem.~11]{SD93} or a reinterpretation of Proposition~\ref{prop:BrXconic} together with the description of the lines given in this section and the construction of Brauer elements via fours given by Swinnerton-Dyer (see for instance \cite[Lem.~10]{SD93} and \cite[Thm 10.]{BBFL07} for the construction of the Brauer elements via fours).
\end{proof}

 \section{A genus 1 fibration and vertical Brauer elements}\label{thegenus1}
 
In what follows we will give a description of the genus 1 fibration $ X_{\bfa} \dashrightarrow \PP^1$ from \cite{VAV14} for which a given Brauer element is vertical. First we recall some basic facts about elliptic surfaces. We then obtain the Brauer element and the genus 1 fibration as in \cite{VAV14} to afterwards reinterpret it in our special setting of surfaces admitting two non-equivalent conic bundles over the ground field. We study how the order of the Brauer group influences the arithmetic of this genus 1 fibration. More precisely, after blowing up the base points of the genus one pencil, we show that the field of definition of its Mordell--Weil group depends on the size of the Brauer group.

\subsection{Background on elliptic surfaces}
Let $k$ be a number field.
\begin{definition}\label{def: ellsurf}
An \emph{elliptic surface} over $k$ is a smooth projective surface $X$ together with a morphism $\mathcal{E}: X \to B$ to some curve $B$  whose generic fibre is a smooth curve of genus $1$, i.e., a genus 1 fibration. If it admits a section we call the fibration \emph{jacobian}. In that case, we fix a choice of section to act as the identity element for each smooth fibre. The set of sections is in one-to-one correspondence with the $k(B)$-points of the generic fibre, hence it has a group structure and it is called the \emph{Mordell--Weil group} of the fibration, or of the surface if there is no doubt on the fibration considered. 
\end{definition}
\begin{remark}
 If $X$ is a rational surface and an elliptic surface, we call it a \emph{rational elliptic surface}.  If the fibration is assumed to be minimal, i.e., no fibre contains $(-1)$-curves as components, then by the adjunction formula the components of reducible fibres are $(-2)$-curves. In that case, the sections are precisely the $(-1)$-curves and the fibration is jacobian over a field of definition of any of the $(-1)$-curves. 
\end{remark}

Given a smooth, projective, algebraic surface $X$ its Picard group has a lattice structure with bilinear form given by the intersection pairing. If $X$ is an elliptic surface then, thanks to the work of Shioda, we know that its Mordell--Weil group also has a lattice structure, with a different bilinear pairing \cite{ShiodaMWL}. Shioda also described the N\'eron--Tate height  pairing via intersections with the zero section and the fibre components. This allows us to determine, for instance, if a given section is of infinite order and the rank of the subgroup generated by a subset of sections.  We give a brief description of the height pairing below.

\begin{definition}
 Let $\mathcal{E}: X\rightarrow B$ be an elliptic surface with Euler characteristic $\chi$. Let $O$ denote the zero section and $P, Q$ two sections of $\mathcal{E}$. The N\'eron--Tate height pairing is given by 
 $$\langle P,Q \rangle= \chi+ P\cdot O +Q\cdot O- P\cdot Q -\sum_{F \in \text{ reducible fibres }} \text{contr}_F(P,Q),$$
 where $\text{contr}_F(P,Q)$ denotes the contribution of the reducible fibre $F$ to the pairing and depends on the type of fibre (see \cite[\S8]{ShiodaMWL} for a list of all possible contributions).
 Upon specializing at $P=Q$ we obtain a formula for the height of a section (point in the generic fibre):
 $$h(P)= \langle P, P \rangle = 2\chi +2P\cdot O - \sum_{F \in \text{ reducible fibres}} \text{contr}_F(P).$$
\end{definition}

\begin{remark}
 The contribution of a reducible fibre depends on the components that $P$ and $Q$ intersect. In this article we deal only with fibres of type $I_4$, thus for the sake of completion and brevity we give only its contribution. Denote by $\Theta_0$ the component that is met by the zero section, $\Theta_1$ and $\Theta_3$ the two components that intersect $\Theta_0$, and let $\Theta_2$ be the component opposite to $\Theta_0$. If $P$ and $Q$ intersect $\Theta_i$ and $\Theta_j$ respectively, with $i\leq j$ then $\text{ contr}_{I_4}(P,Q)=\frac{i(4-j)}{4}$. 
\end{remark}

\subsection{Vertical elements} 

\begin{definition}
 Let $X$ be a smooth surface. Given a genus 1 fibration $\pi: X\rightarrow \PP^1$, the vertical Picard group, denoted by $\Pic_{vert}$, is the subgroup of the Picard group generated by the irreducible components of the fibres of $\pi$. The vertical Brauer group $\Br_{vert}$ is given by the algebras in $\Br k(\PP^1)$ that give Azumaya algebras when lifted to $X$ (see \cite[Def.~3]{Bri06}). 
\end{definition}
  
There is an isomorphism $\Br X/ \Br_0 X \simeq \mathrm{H}^1(k, \Pic \bar{X})$ and, as described by Bright \cite[Prop.4]{Bri06}, a further isomorphism between $B:=\{\mathcal{A} \in \Br k(\PP^1); \pi^* \mathcal{A} \in \Br X\}$ and $\mathrm{H}^1(k, \Pic_{vert})$. Combining these with Theorem \ref{doublefour-Brauer}, allows us to describe vertical Brauer elements as those for which the lines in Theorem \ref{doublefour-Brauer} are fibre components of $\pi$. 

\begin{definition}
We call a Brauer element horizontal w.r.t. $\pi$ if the lines used in Theorem \ref{doublefour-Brauer} to describe it are sections or multisections of $\pi$.
\end{definition}
\begin{remark} As a line cannot be both a fibre component and a (multi)-section simultaneously, a Brauer element that is horizontal cannot be vertical and vice-versa. For a general fibration $\pi$ some Brauer elements might be neither horizontal nor vertical.  
\end{remark}
The following result shows that for a specific elliptic fibration, all Brauer elements are either horizontal or vertical. 

\begin{lemma}\label{lemma: genusonefibration}
Assume that the Brauer group of $X_\bfa$ is non-trivial. Let $F=L_1^{+}+L_2^{+}+M_3^{+}+M_4^{+}$ and $F'=L_1^{-}+L_2^{-}+M_3^{-}+M_4^{-}$. The pencil of hyperplanes spanned by $F$ and $F'$ gives a genus one fibration with exactly two reducible fibres on $X_\bfa$ which are of type $I_4$, for which a non-trivial element of its Brauer group is vertical. The other Brauer elements are horizontal.
\end{lemma}
\begin{proof}
The linear system spanned by $F$ and $F'$ is a subsystem of $|-K_{X_\bfa}|$.  Hence it gives a genus one pencil on $X_\bfa$. Its base points are precisely the four singular points of the following fibres of the conic bundle fibrations:  $L_1^{+}\cup L_1^{-}, L_2^{+}\cup L_2^{-}, M_3^{+}\cup M_3^{-}, M_4^{+} \cup M_4^{-}$. The blow up of these four base points produces a geometrically rational elliptic surface\footnote{not necessarily with a section over the ground field} with two reducible fibres given by the strict transforms of $F$ and $F'$. Since each of the latter is given by four lines in a square configuration and the singular points of this configuration are not blown up, these are of type $I_4$. There are no other reducible fibres as the only $(-2)$-curves are the ones contained in the strict transforms of $F$ and $F'$. Let $\mathcal{E}$ denote the fibration map.
 
 The Azumaya algebra $(f,b)$ with $f$ and $b$ as in Theorem \ref{doublefour-Brauer} taking as double four the components of $F$ and $F'$, gives a Brauer element which is vertical for the genus one fibration $\mathcal{E}$. Indeed, the lines that give such a double four are clearly in $\Pic_{vert}$ and hence the algebra $(f,b)$ lies in the image of $H^1(k, \Pic_{vert}) \rightarrow H^1(k, \bar{X_{\bfa}})$. By \cite[Prop.~4]{Bri06} it gives an element of the form $\mathcal{E}^*\mathcal{A}$, where $\mathcal{A}$ is in $\Br k(\PP^1)$.
 
 The other Brauer elements on $X_\bfa$ are described by double fours, i.e., pairs of sets of four $(-1)$-curves on $X_\bfa$, subject to intersection conditions. Since each component intersects each reducible fibre in exactly one point, after passing to its field of definition, these give sections of the genus one fibration. That is, such Brauer elements are horizontal with respect to this genus one fibration.
 
 %Since the points blown up form two distinct orbits of size two, the genus one fibration admits a section over a degree two extension of the ground field. Fixing one of the exceptional curves of the blow up as the zero section, then its conjugate is a 2-torsion section. The other two exceptional curves correspond to sections of infinite order and are both conjugate and dependent in the Mordell-Weil group, i.e., they add up to a torsion section. The fibration admits a second section of infinite order that is independent of the latter sections. Such a section is given, for example by one of the lines in $X$ that does not pass through any of the blow up points.    
\end{proof}

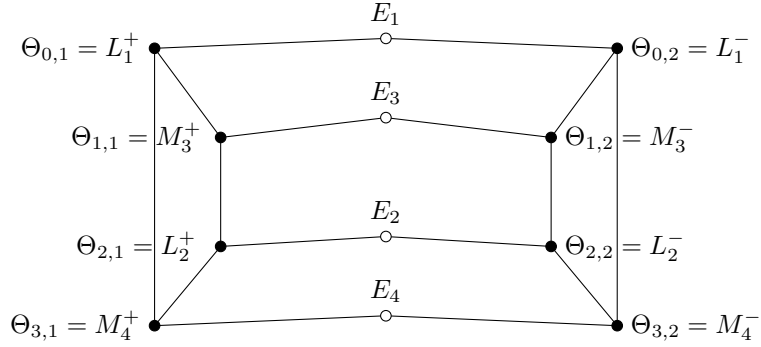
\begin{figure}\label{figure: badfibresgenus1}

 \[
  \begin{tikzpicture}[inner sep=0,x=25pt,y=15pt,font=\footnotesize]

   %\fill [orange!20] (0,-27) rectangle (1,27);  

      \draw (-3.5,3.25) -- (0,3.5);
            \draw (-2.5,1) -- (0,1.5);
            \draw (-2.5,-1.75) -- (0,-1.5);
     \draw (-3.5,-3.75) -- (0,-3.5);

     \draw (3.5,3.25) -- (0,3.5);
            \draw (2.5,1) -- (0,1.5);
            \draw (2.5,-1.75) -- (0,-1.5);
     \draw (3.5,-3.75) -- (0,-3.5);

     \draw (-3.5,3.25) -- (-2.5,1);
            \draw (-2.5,1) -- (-2.5,-1.75);
            \draw (-2.5,-1.75) -- (-3.5,-3.75);
     \draw (-3.5,-3.75) -- (-3.5,3.25);
     
     \draw (3.5,3.25) -- (2.5,1);
            \draw (2.5,1) -- (2.5,-1.75);
            \draw (2.5,-1.75) -- (3.5,-3.75);
     \draw (3.5,-3.75) -- (3.5,3.25);

%\draw[densely dotted] (-3.5) -- (0.5,-27);

%   \draw[fill=white] (0,0) circle (2.5pt);

   %\draw[fill=white] (0,0) circle (2.5pt);

   \filldraw (-3.5,3.25) node[left=5pt]{$\Theta_{0,1}=L_1^{+}$} circle (2pt);

   \filldraw (-2.5,1) node[left=7pt]{$\Theta_{1,1}=M_3^{+}$} circle (2pt);

   \filldraw(-2.5,-1.75) node[left=9pt] {$\Theta_{2,1}=L_2^{+}$} circle (2pt);

      \filldraw (-3.5,-3.75) node[left=5pt]{$\Theta_{3,1}=M_4^{+}$} circle (2pt);

      \filldraw[fill=white] (0, -3.5) node[above=5pt]{$E_4$} circle (2pt);

   \filldraw[fill=white] (0,1.5) node[above=5pt]{$E_3$} circle (2pt);

   \filldraw[fill=white] (0,-1.5) node[above=5pt]{$E_2$} circle (2pt);

      \filldraw[fill=white] (0,3.5) node[above=5pt]{$E_1$} circle (2pt);
 \filldraw (3.5,3.25) node[right=5pt]{$\Theta_{0,2}=L_1^{-}$} circle (2pt);

   \filldraw (2.5,1) node[right=5pt]{$\Theta_{1,2}=M_3^{-}$}circle (2pt);

   \filldraw(2.5,-1.75) node[right=5pt]{$\Theta_{2,2}=L_2^{-}$} circle (2pt);

      \filldraw (3.5,-3.75) node[right=5pt]{$\Theta_{3,2}=M_4^{-}$} circle (2pt);

  \end{tikzpicture}
 \] 
 \caption{The reducible fibres of the genus one fibration $\mathcal{E}$. The eight $\bullet$ denote fibre components and the four $\circ$ denote sections given by the blow up of the 4 base points.}
\end{figure}

\begin{remark}
The genus one fibration for which a Brauer element is vertical described in \cite{VAV14} has in general two reducible fibres given as the union of two geometrically irreducible conics, i.e., they are of type $I_2$. In our setting all the conics are reducible and hence give rise to fibres of type $I_4$. More precisely, let $C_1\cup C_2$ and $C'_1\cup C'_2$ be the two reducible fibres with $C_i$ and $C'_i$ conics, then $C_1\cup C'_1$ is linearly equivalent to one of the fours, say $L_1^{+}\cup L_2^{+}\cup L_1^{-}\cup L_2^{-} $ and $C_1\cup C'_2$ is linearly equivalent to $M_3^{+}\cup M_4^{+} \cup M_3^{-}+M_4^{-}$.
This seems to be very particular of the family considered in this note. More precisely, the presence of two conic bundle structures does not seem to be enough to guarantee that the reducible fibres of the genus one fibration are of type $I_4$. For that one needs that the largest Galois orbit of lines has size at most four and moreover that the field of definition of two of such orbits is the same. 
\end{remark}

\subsection{Mordell--Weil meets Brauer}
In what follows we will keep the letter $\mathcal{E}$ for the genus one fibration on the blow up surface just described. We now give a proof of our main result, Theorem~\ref{thm:MWBrauer}.

% \begin{proposition}\label{MWBrauer}
% Let $X_\bfa$ and $\mathcal{E}$ be as above. Then the following hold:
% 	\begin{enumerate}[label=\emph{(\roman*)}]
% 		\item If $\Br X_\bfa / \Br \QQ \simeq \mathbb{Z}/2\mathbb{Z}$, then the genus 1 fibration $\mathcal{E}$ is an elliptic fibration i.e., admits a section, over a quadratic extension. Moreover, it admits a section of infinite order over a further quadratic extension. The Mordell--Weil group of $\mathcal{E}$ is fully defined over at most a third quadratic extension.
% 		\item If $\Br X_\bfa / \Br \QQ \simeq (\mathbb{Z}/2\mathbb{Z})^2$, then $\mathcal{E}$ is an elliptic fibration with a 2-torsion section and a section of infinite order over $\mathbb{Q}$. Moreover, the full Mordell--Weil group of $\mathcal{E}$ is defined over a quadratic extension.
% 	\end{enumerate}
% \end{proposition}

\begin{proof}
To prove (i) notice that the hypothesis of Proposition~\ref{prop:BrXconic} imply that the four blown up points form two distinct orbits of Galois conjugate points. To exemplify, we work with the genus one fibration given by $F$ and $F'$ as in Lemma~\ref{lemma: genusonefibration}. Let $P_i$ be the intersection point of $L_i^{+}$ and $L_i^{-}$ for $i=1,2$ and that of $M_i^{+}$ and $M_i^{-}$, for $i=3,4$. Denote by $E_i$ the exceptional curve after the blow up of $P_i$. Then $\{E_1,E_2\}$ and $\{E_3,E_4\}$ give two pairs of conjugate sections of $\mathcal{E}$. Moreover, the sections on a pair intersect opposite, i.e., disjoint, components of the fibres given by $F$ and $F'$. Fixing one as the zero section of $\mathcal{E}$, say $E_1$, then a height computation gives that $E_2$ is the 2-torsion section of $\mathcal{E}$. Indeed, as we have fixed $E_1$ as the zero section, the strict transform of $L_1^{+}$ and $L_1^{-}$ are the zero components of the fibres $F$ and $F'$, respectively. We denote them by $\Theta_{0,j}$ with $j=1,2$ respectively. Keeping the standard numbering of the fibre components, the strict transforms of $L_2$ and $L'_2$ are denoted by $\Theta_{2,j}$, with $j=1,2$, respectively. Finally, in this notation, $M_3^{+}$ and $M_3^{-}$ correspond to $\Theta_{1,j}$ while $M_4^{+}$ and $M_4^{-}$ correspond to $\Theta_{3,j}$, for $j=1,2$ respectively.

To compute the height of the section $E_2$ we  need the contribution of each $I_4$ to the pairing which in this case is $1$ (see \cite[\S11]{ShiodaSchuett} for details on the height pairing on elliptic surfaces and the contribution of each singular fibre to it). We have thus
$$\langle E_2, E_2 \rangle= 2-0-1-1=0.$$
In particular, $E_2$ is a torsion section. Since $E_2$ is distinct from the zero section $E_1$ and such fibrations admit torsion of order at most $2$ (see \cite{Persson} for the list of fibre configurations and torsion on rational elliptic surfaces), we conclude that $E_2$ is a 2-torsion section.
The two other conjugate exceptional divisors $E_3$ and $E_4$ give sections of infinite order as one can see, for example, after another height pairing computation.

 To show (ii) it is enough to notice that the hypothesis of Proposition~\ref{prop:BrXconic} imply that the four base points of the linear system spanned by $F$ and $F'$ are defined over $k$. From the discussion above we have that the zero section, the 2-torsion and also a section of infinite order, say $E_3$, are defined over $k$ since each of them is an exceptional curve above a $k$-point. The height matrix of the sections $E_3$ and $E_4$ has determinant zero, hence the section $E_4$ is linearly dependent on $E_3$. Moreover, it follows from the Shioda--Tate formula for $\Pic(X)^{\Gal(\bar{k}/ k)}$ that any section defined over $k$ of infinite order is linearly dependent on $E_3$. Indeed, the rank of the Picard group of the rational elliptic surface is 6 since that of $X_{\bfa}$ has rank 2 and we blow up 4 rational points. The non-trivial components of the two fibres of type $I_4$ give a contribution of 3 to the rank. The other 3 come from the zero section, a smooth fibre and a section of infinite order, say $E_3$. For a second section of infinite order which is independent in the Mordell--Weil group of $E_3$ one can consider the pull-back of a line in $X_\bfa$. The hypothesis on the Brauer group implies that $X_\bfa$ has no line defined over $k$ but each is defined over a quadratic extension. 
\end{proof}

%\begin{remark}
% The arguments above rely only on the presence of two conic bundle structures on $X_\bfa$ and the Galois action on its singular fibres. Thus they can be generalized to any del Pezzo surface of degree four with two conic bundle structures and a $\mathbb{Z}/2\mathbb{Z} \times \mathbb{Z}/2\mathbb{Z}$ action on its singular fibres.
%\end{remark}

\bibliographystyle{amsalpha}{}
\bibliography{bibliography/references}
\end{document}